\documentclass[12pt]{article}
\usepackage{graphicx}\usepackage{float}\usepackage{amsmath}\usepackage{amsfonts}\usepackage{amssymb}\usepackage{color}\usepackage{tikz}
\usepackage[russian, english]{babel}\usepackage{ytableau}
\newtheorem{theorem}{Theorem}

\newtheorem{example}[theorem]{Example}
\newtheorem{lemma}[theorem]{Lemma}\newtheorem{proposition}[theorem]{Proposition}
\newtheorem{remark}[theorem]{Remark}
\newenvironment{proof}[1][Proof]{\noindent\textbf{#1.} }{\ \rule{0.5em}{0.5em}}
\def\myblacksquare{\rule{1.2ex}{1.2ex}}
\def\hh{h}

\newcounter{ourcount}
\newcounter{myenumi}
\begin{document}
$\ $\vskip -3.8cm

\begin{center}
{\Large \textbf{Fusion procedure for the walled Brauer algebra}}

\vspace{.4cm} {\large \textbf{D.V. Bulgakova$^{\diamond\,\dag}$ and O. Ogievetsky$^{\diamond\,\dag\,\ast}$ }}

\vskip .3cm $^{\diamond}$Aix Marseille Universit\'{e}, Universit\'{e} de
Toulon, CNRS, \\ CPT UMR 7332, 13288, Marseille, France

\vskip .05cm $^{\dag}$ I.E.Tamm Dept of Theoretical Physics, Lebedev Physical Institute,\\ 
Leninsky prospect 53, 119991, Moscow, Russia

\vskip .05cm $^{\ast}${Skolkovo Institute of Science and Technology, Moscow, Russia}
\end{center}

\begin{abstract}\noindent
We establish two versions of the fusion procedure for the walled Brauer algebras. In each of them, a complete system of primitive pairwise orthogonal idempotents for  the walled Brauer algebra
is constructed by consecutive evaluations of a rational function in several variables on contents of standard walled tableaux. 
\end{abstract}

\section{Introduction}
The main interest in the walled Brauer algebra ${\sf B}_{r,s}(\delta)$ is related to its role of 
the centralizer of the diagonal action of the group $GL(V)$ on the mixed tensor space $V^{\otimes r}\otimes (V^*)^{\otimes s}$, see \cite{K,T,BCHLLS,Ni}. As the centralizer, the walled Brauer algebra acts irreducibly   
on the space of multiplicities in the decomposition of the tensor product $V^{\otimes r}\otimes (V^*)^{\otimes s}$ into a sum of $GL(V)$-irreducibles, so we 
need to understand well the basis in every irreducible representation of the algebra ${\sf B}_{r,s}(\delta)$. A fusion procedure gives a construction of the maximal family of 
pairwise orthogonal minimal idempotents in the algebra, and therefore, provides a way to understand bases in the irreducible representations of the algebra ${\sf B}_{r,s}(\delta)$. 

\vskip .1cm
The fusion procedure (for the symmetric group) originates in the work of Jucys \cite{Ju}, see also the subsequent works \cite{Ch,Na,GP}. 
A simplified version of the fusion procedure for the symmetric group involving the consecutive evaluations was suggested by Molev in \cite{Mo}. 
Later the analogues of this simplified fusion procedure were suggested for the Hecke algebra \cite{IM}, for the Brauer algebra \cite{IM,IMO}, 
for the complex reflection groups of type $G(m,1,n)$, for the cyclotomic Hecke algebras \cite{OP3,OP4}, for the cyclotomic Brauer algebras \cite{C}, for the Birman--Murakami--Wenzl algebras \cite{IMO2}. In \cite{OP5} the fusion procedure was applied for the calculation of weights of certain Markov traces on the cyclotomic Hecke algebras.

\vskip .1cm
The walled Brauer algebras form a tower which is an important ingredient of our presentation. The first $r$ floors 
of the tower reproduce the tower of the symmetric groups. Essentially new effects start to appear at the $(r+1)$-st 
floor, after crossing the wall. 

\vskip .1cm
In this work we develop the fusion procedure, in the spirit of \cite{Mo}, for the walled Brauer algebra. That is, we exhibit a `universal' rational function with values in the walled Brauer algebra whose consecutive evaluations at the contents of the appropriately generalized `walled tableaux' give a complete system of primitive idempotents. This function is a product of a prefactor $z_T$, which is a usual rational function, and a universal rational function $\Psi_{r,s}$, 
see Section \ref{maosrere} for details. We find the minimal collection ${\sf ze}_T$ of factors in $z_T$ such that the
consecutive evaluations of the product ${\sf ze}_T\Psi_{r,s}$ are finite. To describe the function ${\sf ze}_T$ we adapt
the notion of `exponents' introduced in \cite {IM} for the full Brauer algebra to the situation of the walled Brauer algebra.
For the walled Brauer algebra the exponents take a simpler form than for the full Brauer algebra; the exponents are expressed in  terms of the Laplacian of a function which counts the lengths of the diagonal in diagrams. We exhibit a version of the fusion procedure which depends on a free parameter. 

\section{Walled Brauer algebra}
\subsection{Definition}
Let $r$ and $s$ be non-negative integers. We denote by $p^u_{r,s}$ and $p^d_{r,s}$ two sets, each consisting of $r+s$ points aligned horizontally  on the plane; the points in the set $p^d_{r,s}$
are placed under the points in the set $p^u_{r,s}$. The left $r$ points in the sets  $p^u_{r,s}$ and $p^d_{r,s}$ we denote by $a^u$ and $a^d$, the right $s$ points -- by $b^u$ and $b^d$. The left 
points are separated by the wall from the right points. 
A walled $(r,s)$-diagram $d$ is a bijection between the sets $a^u\cup b^d$ and $a^d\cup b^u$. In particular the number of walled $(r,s)$-diagrams is equal to $(r+s)!$. 
A walled $(r,s)$-diagram is visualised by placing the edges between the corresponding points 
of $a^u\cup b^d$ and $a^d\cup b^u$. 

\vskip .2cm
Let $\delta$ be a complex parameter. The walled Brauer algebra ${\sf B}_{r,s}(\delta)$ is defined as the $\mathbb{C}$-linear span of the walled $(r,s)$-diagrams with the multiplication defined as follows.
The product of two diagrams $\Upsilon_1$ and $\Upsilon_2$ is determined by placing $\Upsilon_1$ above $\Upsilon_2$ and identifying the vertices of the bottom row of $\Upsilon_1$ with the corresponding vertices in the top row of $\Upsilon_2$. Let $\ell$ be the number of closed loops so obtained. 
The product $\Upsilon_1\Upsilon_2$ is given by $\delta^{\ell}$ times the resulting diagram with loops omitted.  

\vskip .2cm
Let $s_i$, $1\leqslant i<r$ or $r<i<r+s$, and $d$ denote the following walled $(r,s)$-diagrams (the vertical dotted line represents the wall): 

\begin{align*}
&\begin{array}{c}
\begin{tikzpicture}
\draw (0,0) --(0,1.2);\draw (1.5,0) --(2,1.2);\draw (1.5,1.2) --(2,0);\draw (3.5,0) --(3.5,1.2);\draw [dotted] (4,-.2) --(4,1.4);\draw (4.5,0) --(4.5,1.2);\draw (6,0) --(6,1.2);
\node at (.8,1.2) {$\dots$};\node at (.8,0) {$\dots$};\node at (2.8,1.2) {$\dots$};\node at (2.8,0) {$\dots$};\node at (5.3,1.2) {$\dots$};\node at (5.3,0) {$\dots$};
\node at (0,-.4) {$1$};\node at (1.3,-.4) {$i$};\node at (2.2,-.42) {$i+1$};\node at (3.5,-.43) {$r$};\node at (4.8,-.4) {$r+1$};\node at (6.3,-.4) {$r+s$};
\node at (8,.5) {$,\ \ \ i=1,\dots,r-1$};\node at (-1.2,.5) {$s_i:=$};
\end{tikzpicture}
\end{array}\\
&\begin{array}{c}
\begin{tikzpicture}
\draw (6,0) --(6,1.2);\draw (4.5,0) --(4,1.2);\draw (4.5,1.2) --(4,0);\draw (2.5,0) --(2.5,1.2);\draw [dotted] (2,-.2) --(2,1.4);\draw (1.5,0) --(1.5,1.2);\draw (0,0) --(0,1.2);
\node at (5.2,1.2) {$\dots$};\node at (5.2,0) {$\dots$};\node at (3.2,1.2) {$\dots$};\node at (3.2,0) {$\dots$};\node at (0.7,1.2) {$\dots$};\node at (0.7,0) {$\dots$};
\node at (0,-.4) {$1$};\node at (4.7,-.42) {$i+1$};\node at (3.8,-.4) {$i$};\node at (2.6,-.4) {$r+1$};\node at (1.5,-.43) {$r$};\node at (6.3,-.42) {$r+s$};
\node at (8.7,.5) {$,\ \ \ i=r+1,\dots,r+s-1$};\node at (-1.2,.5) {${s}_i:=$};
\end{tikzpicture}
\end{array}\\
&\begin{array}{c}
\begin{tikzpicture}
\draw (0,0) --(0,1.2);\draw (1.5,0) --(1.5,1.2);\draw (2.1,0) .. controls (2.2,0.4) and (2.7,0.4) .. (2.8,0);\draw [dotted] (2.45,-.2) --(2.45,1.4);
\draw (2.1,1.2) .. controls (2.2,0.8) and (2.7,0.8) .. (2.8,1.2);\draw (3.4,0) --(3.4,1.2);\draw (4.9,0) --(4.9,1.2);\node at (.8,1.2) {$\dots$};
\node at (.8,0) {$\dots$};\node at (4.2,1.2) {$\dots$};\node at (4.2,0) {$\dots$};\node at (0,-.4) {$1$};\node at (2.0,-.42) {$r$};\node at (2.95,-.40) {$r+1$};
\node at (5.0,-.40) {$r+s$};\node at (-1.2,.5) {$d:=$};
\end{tikzpicture}
\end{array}\\
\end{align*}
The walled Brauer algebra ${\sf B}_{r,s}(\delta)$ is generated by the elements $s_i$, $1\leqslant i<r$ or $r<i<r+s$, and $d$, with the 
defining relations, see, e.g., \cite{BS,JK}, 
$$\begin{array}{c}
s_i^2=1\ ,\ d^2=\delta d\ ,\\[.5em]
s_is_{i+1}s_i=s_{i+1}s_is_{i+1}\ ,\\[.5em]
s_i s_j=s_j s_i \ \ \text{if}\ \  |i-j|>1\ ,\\[.5em]
d s_{r\pm 1} d= d\ \text{ and } \ ds_i=s_id \ \ \text{if}\ \  i\neq r\pm1\ , \\[.5em]
d s_{r+1} s_{r-1} d s_{r-1}=d s_{r+1} s_{r-1} d s_{r+1}\ ,\\[.5em]
s_{r-1} d {s}_{r+1} s_{r-1} d={s}_{r+1} d {s}_{r+1} s_{r-1} d\ .
\end{array}$$

The algebra ${\sf B}_{r,s}(\delta)$ admits an anti-automorphism $\iota$, which acts as identity on the generators,
\begin{equation}\label{antiai} \iota(s_i)=s_i\ ,\ \iota(d)=d\ ,\ \iota(xy)=\iota(y)\iota(x)\ .\end{equation}

\vskip .2cm
The subalgebra, generated by the elements $s_i$, $1\leqslant i<r$, of the algebra ${\sf B}_{r,s}(\delta)$ 
is isomorphic to the group ring $\mathbb{C}[\mathbb{S}_r]$ of the symmetric group; 
the subalgebra, generated by the elements ${s}_i$, $r<i<r+s$, is isomorphic to the group ring $\mathbb{C}[\mathbb{S}_s]$. 

\vskip .2cm
It is known, see \cite{CVDM}, Theorem 6.3, that the walled Brauer algebra ${\sf B}_{r,s}(\delta)$ is semisimple if and
only if one of the following conditions holds:
\vskip .1cm\par $r = 0$ or $s = 0$,
\vskip .1cm\par $\delta \notin\mathbb{Z}$,
\vskip .1cm\par $\vert \delta\vert > r + s - 2$,
\vskip .1cm\par $\delta = 0$ and $(r, s)\in \{(1, 2),(1, 3),(2, 1),(3, 1)\}$.

\vskip .2cm
In the sequel we always assume that $\delta$ is generic, that is, the walled Brauer algebra is semisimple.

\vskip .2cm
Let $\mathsf{A}_k$ be the subalgebra in the algebra ${\sf B}_{r,s}(\delta)$ generated by the walled $(r,s)$-diagrams non-trivial only at the first $k$ sites of the sets
$p^u_{r,s}$ and $p^d_{r,s}$ (that is, to the right of the $k$-th site the diagram has only vertical segments).  For $k\leqslant r$, the algebra $\mathsf{A}_k$
is isomorphic to ${\sf B}_{k,0}(\delta)\cong \mathbb{C}[\mathbb{S}_k]$ while for $r<k\leqslant r+s$ the algebra $\mathsf{A}_k$
is isomorphic to ${\sf B}_{r,k-r}(\delta)$. The algebras $\mathsf{A}_k$, $0\leqslant k\leqslant r+s$, form an ascending chain of algebras
\begin{equation}\label{ascala}\mathbb{C}\equiv\mathsf{A}_0\subset \mathsf{A}_1\subset\ldots\subset\mathsf{A}_{r+s}\equiv {\sf B}_{r,s}(\delta)\ .\end{equation}
The Bratteli diagram for this chain plays an important role in the representation theory of the walled Brauer algebras and, in particular, in our fusion procedures. 

\vskip .2cm
If $r=0$ or $s=0$ the algebra ${\sf B}_{r,s}(\delta)$ is the group algebra of the symmetric group, for which the fusion procedure is well established, so we shall always assume that both $r$ and $s$ are different from 0. 

\vskip .2cm
To shorten the formulation of our results it is convenient to introduce the following function on the set $\{ 1,\dots ,r+s\}$
\begin{equation}\label{funepsi} \varepsilon (j):=\left\{ \begin{array}{l}0\ \text{ if } \ j\leqslant r\ ,\\[.4em]
1\ \text{ if }\  j> r\ .\end{array}\right.\end{equation}

\subsection{Jucys--Murphy elements}
The Jucys--Murphy elements for the walled Brauer algebras are adapted to the chain (\ref{ascala}).

\vskip .2cm
Let $s_{i,k}$ where $1\leqslant i<k\leqslant r$ or $ r+1\leqslant i<k\leqslant r+s$ and $d_{i,k}$ where $1\leqslant i\leqslant r<k\leqslant r+s$
denote the following walled $(r,s)$-diagrams:

\vskip .2cm
\begin{tikzpicture}
\draw (0,0) --(0,1.2);\draw (1,0) --(2.5,1.2);\draw (1,1.2) --(2.5,0);\draw (3.5,0) --(3.5,1.2);\draw [dotted] (4,-.2) --(4,1.4);\draw (4.5,0) --(4.5,1.2);\draw (6,0) --(6,1.2);
\node at (.55,1.2) {$\dots$};\node at (.55,0) {$\dots$};\node at (1.8,1.2) {$\dots$};\node at (1.8,0) {$\dots$};\node at (3,1.2) {$\dots$};\node at (3,0) {$\dots$};
\node at (5.3,1.2) {$\dots$};\node at (5.3,0) {$\dots$};
\node at (0,-.4) {$1$};\node at (0.95,-.4) {$i$};\node at (2.55,-.42) {$k$};\node at (3.5,-.43) {$r$};\node at (4.8,-.4) {$r+1$};\node at (6.3,-.4) {$r+s$};
\node at (7.9,.5) {$,$}; \node at (-1.2,.5) {$s_{i,k}:=$};
\end{tikzpicture}

\vskip .4cm
\begin{tikzpicture}
\draw (6,0) --(6,1.2);\draw (5,0) --(3.5,1.2);\draw (5,1.2) --(3.5,0);\draw (2.5,0) --(2.5,1.2);\draw [dotted] (2,-.2) --(2,1.4);\draw (1.5,0) --(1.5,1.2);\draw (0,0) --(0,1.2);
\node at (5.5,1.2) {$\dots$};\node at (5.5,0) {$\dots$};\node at (4.3,1.2) {$\dots$};\node at (4.3,0) {$\dots$};\node at (3.05,1.2) {$\dots$};\node at (3.05,0) {$\dots$};
\node at (0.7,1.2) {$\dots$};\node at (0.7,0) {$\dots$};
\node at (0,-.4) {$1$};\node at (5,-.42) {$k$};\node at (3.5,-.4) {$i$};\node at (2.6,-.4) {$r+1$};\node at (1.5,-.43) {$r$};\node at (6.2,-.42) {$r+s$};
\node at (8.6,.5) {$,$}; \node at (-1.2,.5) {${s}_{i,k}:=$};
\end{tikzpicture}

\vskip .4cm
\begin{tikzpicture}
\draw (0,0) --(0,1.2);\draw (3,0) --(3,1.2);\draw (1.5,0) .. controls (2.4,0.55) and (4.5,0.55) .. (5.4,0);\draw [dotted] (3.45,-.2) --(3.45,1.4);
\draw (1.5,1.2) .. controls (2.4,0.65) and (4.5,0.65) .. (5.4,1.2);\draw (3.9,0) --(3.9,1.2);\draw (6.9,0) --(6.9,1.2);\node at (2.4,1.2) {$\dots$};\node at (2.4,0) {$\dots$};
\node at (4.6,1.2) {$\dots$};\node at (4.6,0) {$\dots$};\node at (0.9,1.2) {$\dots$};\node at (0.9,0) {$\dots$};\node at (6.1,1.2) {$\dots$};\node at (6.1,0) {$\dots$};
\node at (0,-.4) {$1$};\node at (1.5,-.4) {$i$};\node at (2.9,-.42) {$r$};\node at (4.1,-.40) {$r+1$};\node at (5.4,-.4) {$k$};\node at (7,-.40) {$r+s$};
\node at (8.5,.5) {$.$}; \node at (-1.2,.5) {$d_{i,k}:=$};
\end{tikzpicture}

\vskip .2cm
\noindent In particular, $s_i=s_{i,i+1}$, $1\leqslant i<r$ or $r<i<r+s$, $d=d_{r,r+1}$.

\vskip .2cm
In terms of generators, the elements $s_{i,k}$ and $d_{i,k}$ can be written as 
$$\begin{array}{c}
    s_{i,k}=s_i s_{i+1}\dots s_{k-2}s_{k-1}s_{k-2}\dots s_{i+1}s_i\,,\\[.6em]
     d_{i,k}=s_i s_{i+1}\dots s_{r-1}{s}_{k-1}{s}_{k-2}\dots {s}_{r+1}d{s}_{r+1}\dots{s}_{k-2}{s}_{k-1}s_{r-1}\dots s_{i+1}s_i\,.
\end{array}$$

The Jucys-Murphy elements for the walled Brauer algebra are, see \cite{BS,SS,JK},
$$x_{k}:=\left\{\begin{array}{ll}
    \hspace{.34cm}{\displaystyle \sum_{i=1}^{k-1} }\, s_{i,k}      & \text{if } 1\leqslant k\leqslant r\ ,\\[.6cm]
    -{\displaystyle \sum_{i=1}^{r} }\, d_{i,k} +{\displaystyle\sum_{i=r+1}^{k-1} }\, {s}_{i,k}+\delta  & \text{if }  k\geqslant r+1\ .
  \end{array}\right.$$
One checks that the element $x_k$ commutes with any element of the subalgebra $\mathsf{A}_{k-1}$, in the notation (\ref{ascala}) for the tower.
This implies that the elements $x_1$,$\dots$, $x_{r+s}$ of ${\sf B}_{r,s}(\delta)$ pairwise commute.

\vskip .2cm
Moreover, it follows from the representation theory of the walled Brauer algebras that the subalgebra generated by the elements 
$x_1,\ldots ,x_{r+s}$ is a maximal commutative subalgebra of the algebra ${\sf B}_{r,s}(\delta)$. 

\vskip .2cm
We note that the walled Brauer algebra ${\sf B}_{r,s}(\delta)$ is naturally a subalgebra of the Brauer algebra  ${\sf B}_{r+s}(\delta)$. 
Interestingly, the Jucys--Murphy elements for the walled Brauer algebra are obtained by omitting in the Jucys--Murphy elements for the Brauer algebra 
the summands corresponding to the diagrams which do not exist in the walled Brauer algebra.

\subsection{Representations}
We will identify each partition $\gamma$ with its diagram so that if the parts of $\gamma$ are $\gamma_1,\gamma_2,\dots$, 
$\gamma_1\geqslant\gamma_2\geqslant\dots$, then the corresponding diagram is a left-justified array of rows of boxes containing $\gamma_1$ boxes in the top row, $\gamma_2$ boxes in the second row, etc. The box in the row $i$ and column $j$ will be denoted by $(i,j)$.

\vskip .2cm
We let 
$$\text{$A(\gamma)$ be the set of all addable cells}$$
and 
$$\text{$R(\gamma)$ the set of all removable cells}$$
of the diagram $\gamma$. 

\vskip .2cm
A bipartition $\Lambda$ is a pair of partitions $\Lambda=(\lambda_L,\lambda_R)$. Let $\mathcal{D}$ be the set of all bipartitions. For each integer $0\leqslant f\leqslant \min(r,s)$, let 
$$\mathcal{D}_{r,s}(f):=\{ \Lambda\in\mathcal{D} \colon |\lambda_L|=r-f\, ,\, |\lambda_R|=s-f\}\ ,$$
where $|\lambda |$ denotes the sum of the parts of $\lambda$, and 
$$   \mathcal{D}_{r,s}:=\bigcup_{f=0}^{\min(r,s)}\mathcal{D}_{r,s}(f)\ .$$
The irreducible representations of the walled Brauer algebra ${\sf B}_{r,s}(\delta)$ are indexed by the elements of the set $\mathcal{D}_{r,s}$, see \cite{CVDM}, Theorem 2.7. 
We shall denote by $V_\Lambda$ the representation corresponding to a bipartition $\Lambda\in \mathcal{D}_{r,s}$.

\vskip .2cm
 The branching rules for the ascending chain (\ref{ascala}) are simple. Here is the figure showing the Bratteli diagram 
for the chain on the example of the walled Brauer algebra ${\sf B}_{2,2}(\delta)$.

\begin{figure}[H]
\centering
\begin{tikzpicture}
\node  at (0,0) {$(\varnothing,\varnothing)$};\draw [thick] (0.64,0) -- (1.56,0);\node  at (1.7,0) {(};\draw[fill] (1.9,0) circle [radius=0.08];\node  at (2.3,0) {$\ ,\varnothing )$};
\node  at (3.7,-1) {(};\draw[fill] (3.9,-1) circle [radius=0.08];\draw[fill] (4.1,-1) circle [radius=0.08];\node  at (4.5,-1) {$\ ,\varnothing )$};
\node  at (3.7,1) {(};\draw[fill] (3.9,1.1) circle [radius=0.08];\draw[fill] (3.9,0.9) circle [radius=0.08];\node  at (4.3,1) {$\ \ ,\varnothing )$};
\draw [thick] (2.75,0.1) -- (3.58,0.75);\draw [thick] (2.75,-0.1) -- (3.58,-0.75);
\node  at (5.75,0) {(};\draw[fill] (5.95,0) circle [radius=0.08];\node  at (6.35,0) {$\ ,\varnothing )$};
\node  at (5.75,-2) {(};\draw[fill] (5.95,-2) circle [radius=0.08];\draw[fill] (6.15,-2) circle [radius=0.08];\draw[fill] (6.55,-2) circle [radius=0.08];
\node  at (6.75,-2) {)};\node  at (6.3,-2.1) {$\ ,$};
\node  at (5.75,2) {(};\draw[fill] (5.95,2.1) circle [radius=0.08];\draw[fill] (5.95,1.9) circle [radius=0.08];\draw[fill] (6.35,2) circle [radius=0.08];
\node  at (6.55,2) {)};\node  at (6.08,1.87) {$\ \, ,$};
\draw [thick] (4.8,0.75) -- (5.63,0.1);\draw [thick] (4.88,-0.8) -- (5.63,-0.1);\draw [thick] (4.8,1.2) -- (5.6,1.9);\draw [thick] (4.86,-1.2) -- (5.6,-1.9);
\node  at (8.9,-0.7) {$(\varnothing,\varnothing)$};\node  at (8.4,0.7) {(};\draw[fill] (8.6,0.7) circle [radius=0.08];\draw[fill] (8.97,0.7) circle [radius=0.08];
\node  at (9.17,0.7) {)};\node  at (8.73,0.6) {$\ ,$};
\node  at (8.4,-2) {(};\draw[fill] (8.6,-2) circle [radius=0.08];\draw[fill] (8.8,-2) circle [radius=0.08];\draw[fill] (9.2,-2) circle [radius=0.08];
\draw[fill] (9.4,-2) circle [radius=0.08];\node  at (9.6,-2) {)};\node  at (8.95,-2.1) {$\ ,\,$};
\node  at (8.4,-3.2) {(};\draw[fill] (8.6,-3.2) circle [radius=0.08];\draw[fill] (8.8,-3.2) circle [radius=0.08];\draw[fill] (9.2,-3.3) circle [radius=0.08];
\draw[fill] (9.2,-3.1) circle [radius=0.08];\node  at (9.43,-3.2) {)};\node  at (8.95,-3.3) {$\ ,\,$};
\node  at (8.4,3.2) {(};\draw[fill] (8.6,3.3) circle [radius=0.08];\draw[fill] (8.6,3.1) circle [radius=0.08];\draw[fill] (9,3.1) circle [radius=0.08];
\draw[fill] (9,3.3) circle [radius=0.08];\node  at (9.22,3.2) {)};\node  at (8.73,3.07) {$\ ,$};
\node  at (8.4,2) {(};\draw[fill] (8.6,2.1) circle [radius=0.08];\draw[fill] (8.6,1.9) circle [radius=0.08];\draw[fill] (9,2) circle [radius=0.08];
\draw[fill] (9.2,2) circle [radius=0.08];\node  at (9.4,2) {)};\node  at (8.73,1.87) {$\ ,$};
\draw [thick] (6.76,0.12) -- (8.3,0.7);\draw [thick] (6.76,-0.1) -- (8.3,-0.65);\draw [thick] (6.7,1.9) -- (8.3,0.8);\draw [thick] (6.7,2) -- (8.3,2);
\draw [thick] (6.7,2.1) -- (8.3,3.1);\draw [thick] (6.9,-1.9) -- (8.3,0.6);\draw [thick] (6.9,-2) -- (8.3,-2);\draw [thick] (6.9,-2.1) -- (8.3,-3.2);
\end{tikzpicture}
\caption{Bratteli diagram} \label{fig:1}
\end{figure}

\subsection{A basis in $V_\Lambda$}
 The following construction is given in \cite{JK}. Let $\Lambda=(\lambda_L,\lambda_R)$ be a bipartition. A {\it standard walled  $\Lambda$-tableau} is a sequence $T=(\Lambda^{(0)},\dots,\Lambda^{(r+s)})$ of bipartitions such that $\Lambda^{(0)}=(\varnothing ,\varnothing)$, $\Lambda^{(r+s)}=\Lambda$ and for each $t=1,\dots,r+s$ the bipartition $\Lambda^{(t)}$ is obtained from $\Lambda^{(t-1)}$ by
 \begin{itemize}
\item adding a box to the first diagram in the bipartition when $1\leqslant t\leqslant r$; 
\item adding a box to the second diagram or removing a box from the first diagram when $r+1\leqslant t\leqslant r+s$. 
\end{itemize}
Thus, $T$ represents a path in the Bratteli diagram of the algebra ${\sf B}_{r,s}(\delta)$. 

\vskip .2cm
We say that $r+s$ is the length of $T$. We write $U\nearrow T$ if the standard walled tableau $U$ of length $r+s-1$ is obtained by removing the last entry 
$\Lambda^{(r+s)}$ from the sequence $T$. 

\vskip .2cm
We shall denote by $\mathcal{T}_\Lambda$ the set of all standard walled $\Lambda$-tableaux.

\vskip .2cm
As always, in the situation when the branching rules for a Bratteli diagram are simple, a path $T\in\mathcal{T}_\Lambda$ defines a one-dimensional 
subspace $V_T$ in the space $V_\Lambda$. Choose an arbitrary non-zero vector $v_T\in V_T$. The set $\{ v_T\}$, $T\in\mathcal{T}_\Lambda$, is a basis 
of the space $V_\Lambda$. 

\vskip .2cm
To each standard walled  tableau $T$ we attach its sequence of {\it contents} $(c_1(T),\dots,c_{r+s}(T))$, where
\begin{itemize}
\item if $ 1\leqslant t\leqslant r$ and $\Lambda^{(t)}$ is obtained from $\Lambda^{(t-1)}$ by adding a box $(i,j)$ to the first diagram in the bipartition then
\begin{equation}\label{c1}c_{t}(T):=j-i\ ,\end{equation}
\item if $t\geqslant r+1$ and $\Lambda^{(t)}$ is obtained from $\Lambda^{(t-1)}$ by removing a box $(i,j)$ from the first diagram then
\begin{equation}\label{c3}c_{t}(T):=-(j-i)\ , \end{equation}
\item and, if $t\geqslant r+1$ and $\Lambda^{(t)}$ is obtained from $\Lambda^{(t-1)}$ by adding a box $(i,j)$ to the second diagram then
\begin{equation}\label{c2}c_{t}(T):=  (j-i)+\delta\ .\end{equation}
\end{itemize}

\vskip .2cm
It is convenient to decorate the Bratteli diagram, writing at each edge of the path $T$ the corresponding content, as shown below on our example of the algebra 
${\sf B}_{2,2}(\delta)$.
  
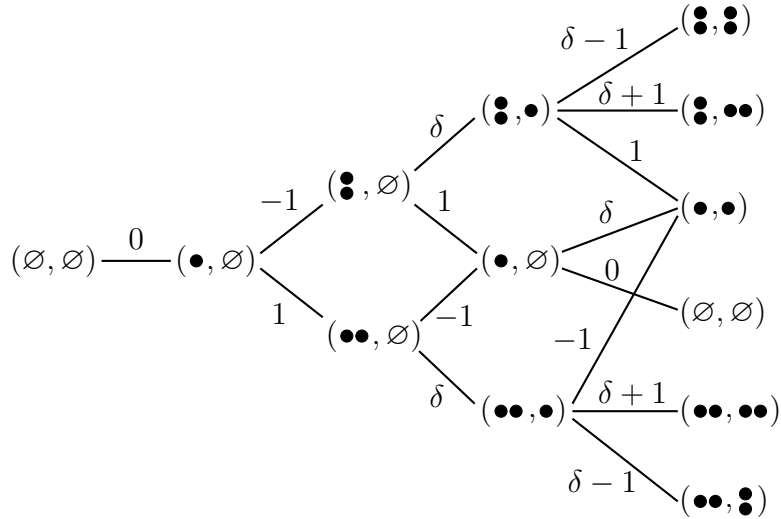
\begin{figure}[H]
\centering
\begin{tikzpicture}
\node  at (0,0) {$(\varnothing,\varnothing)$};\draw [thick] (0.64,0) -- (1.56,0);\node  at (1.7,0) {(};\draw[fill] (1.9,0) circle [radius=0.08];\node  at (2.3,0) {$\ ,\varnothing )$};
\node  at (3.7,-1) {(};\draw[fill] (3.9,-1) circle [radius=0.08];\draw[fill] (4.1,-1) circle [radius=0.08];\node  at (4.5,-1) {$\ ,\varnothing )$};
\node  at (3.7,1) {(};\draw[fill] (3.9,1.1) circle [radius=0.08];\draw[fill] (3.9,0.9) circle [radius=0.08];\node  at (4.3,1) {$\ \ ,\varnothing )$};
\draw [thick] (2.75,0.1) -- (3.58,0.75);\draw [thick] (2.75,-0.1) -- (3.58,-0.75);
\node  at (5.75,0) {(};\draw[fill] (5.95,0) circle [radius=0.08];\node  at (6.35,0) {$\ ,\varnothing )$};\node  at (5.75,-2) {(};\draw[fill] (5.95,-2) circle [radius=0.08];
\draw[fill] (6.15,-2) circle [radius=0.08];\draw[fill] (6.55,-2) circle [radius=0.08];\node  at (6.75,-2) {)};\node  at (6.3,-2.1) {$\ ,$};
\node  at (5.75,2) {(};\draw[fill] (5.95,2.1) circle [radius=0.08];\draw[fill] (5.95,1.9) circle [radius=0.08];\draw[fill] (6.35,2) circle [radius=0.08];\node  at (6.55,2) {)};
\node  at (6.08,1.87) {$\ \, ,$};\draw [thick] (4.8,0.75) -- (5.63,0.1);\draw [thick] (4.88,-0.8) -- (5.63,-0.1);\draw [thick] (4.8,1.2) -- (5.6,1.9);
\draw [thick] (4.86,-1.2) -- (5.6,-1.9);\node  at (8.9,-0.7) {$(\varnothing,\varnothing)$};\node  at (8.4,0.7) {(};\draw[fill] (8.6,0.7) circle [radius=0.08];
\draw[fill] (8.97,0.7) circle [radius=0.08];\node  at (9.17,0.7) {)};\node  at (8.73,0.6) {$\ ,$};
\node  at (8.4,-2) {(};\draw[fill] (8.6,-2) circle [radius=0.08];\draw[fill] (8.8,-2) circle [radius=0.08];\draw[fill] (9.2,-2) circle [radius=0.08];
\draw[fill] (9.4,-2) circle [radius=0.08];\node  at (9.6,-2) {)};\node  at (8.95,-2.1) {$\ ,\,$};
\node  at (8.4,-3.2) {(};\draw[fill] (8.6,-3.2) circle [radius=0.08];\draw[fill] (8.8,-3.2) circle [radius=0.08];\draw[fill] (9.2,-3.3) circle [radius=0.08];
\draw[fill] (9.2,-3.1) circle [radius=0.08];\node  at (9.43,-3.2) {)};\node  at (8.95,-3.3) {$\ ,\,$};
\node  at (8.4,3.2) {(};\draw[fill] (8.6,3.3) circle [radius=0.08];\draw[fill] (8.6,3.1) circle [radius=0.08];\draw[fill] (9,3.1) circle [radius=0.08];
\draw[fill] (9,3.3) circle [radius=0.08];\node  at (9.22,3.2) {)};\node  at (8.73,3.07) {$\ ,$};
\node  at (8.4,2) {(};\draw[fill] (8.6,2.1) circle [radius=0.08];\draw[fill] (8.6,1.9) circle [radius=0.08];\draw[fill] (9,2) circle [radius=0.08];
\draw[fill] (9.2,2) circle [radius=0.08];\node  at (9.4,2) {)};\node  at (8.73,1.87) {$\ ,$};
\draw [thick] (6.76,0.12) -- (8.3,0.7);\draw [thick] (6.76,-0.1) -- (8.3,-0.65);\draw [thick] (6.7,1.9) -- (8.3,0.8);\draw [thick] (6.7,2) -- (8.3,2);
\draw [thick] (6.7,2.1) -- (8.3,3.1);\draw [thick] (6.9,-1.9) -- (8.3,0.6);\draw [thick] (6.9,-2) -- (8.3,-2);\draw [thick] (6.9,-2.1) -- (8.3,-3.2);
\node  at (1.1,0.3) {$0$};\node  at (3,0.8) {$-1$};\node  at (3,-0.7) {$1$};\node  at (5.1,1.8) {$\delta$};\node  at (5.1,-1.8) {$\delta$};
\node  at (5.2,0.8) {$1$};\node  at (5.33,-0.75) {$-1$};\node  at (7.2,2.95) {$\delta-1$};\node  at (7.7,2.22) {$\delta+1$};\node  at (7.75,1.47) {$1$};
\node  at (7.35,0.65) {$\delta$};\node  at (7.43,-0.12) {$0$};\node  at (7.7,-1.76) {$\delta+1$};\node  at (7.3,-2.95) {$\delta-1$};\node  at (6.9,-1.05) {$-1$};
\end{tikzpicture}
\caption{Paths and contents} \label{fig:2}
\end{figure}

Clearly, a path $T$ can be reconstructed from its sequence of contents. 

\vskip .2cm
We encode the sequence $T$ of bipartitions in the following way. We first associate to the sequence $T$ three Young diagrams   
$\lambda'(T)$, $\nu(T)$ and $\lambda''(T)$ such that $\nu(T)\subseteq \lambda'(T)$. The  
diagram $\lambda'(T):=\lambda_L^{(r)}$ is the left diagram in the bipartition $\Lambda^{(r)}$. 
The diagram $\nu(T):=\lambda_L^{(r+s)}$ is the left final diagram and the diagram
$\lambda''(T):=\lambda_R^{(r+s)}$ is the right final diagram in the bipartition $\Lambda^{(r+s)}$. We call $$D_T  :=[\lambda'(T),\,\nu(T),\,\lambda''(T) ]$$ the {\it triple diagram}
corresponding to the path $T$.

\vskip .2cm
Next, we fill the boxes of the diagrams  $\lambda'(T)$, $\lambda''(T)$ and the set-theoretical difference $\lambda'(T)\setminus\nu(T)$.
Exactly as for the symmetric group, the boxes of the diagram $\lambda'(T)$ are filled with numbers $1,\dots ,r$, 
representing the order in which the boxes were added in the sequence $(\lambda_L^{(0)},\dots,\lambda_L^{(r)})$. The boxes of the union 
$\bigl( \lambda'(T)\setminus\nu(T)\bigr)\sqcup\lambda''(T)$ are filled with numbers $r+1,\dots ,r+s$ 
in the order in which the boxes were removed or added in the sequence $(\Lambda^{(r+1)},\dots,\Lambda^{(r+s)})$. The resulting filling we call the 
{\it standard triple tableau} $W_T$ corresponding to the path $T$.

\vskip .2cm
It is straightforward to see that the correspondence between the set of all paths and the set of all standard triple tableaux is one to one. 

\vskip .2cm
We visualize the standard walled tableau by putting the numbers corresponding to the filling of $\lambda'(T)$ in the upper left corner of boxes and the numbers
corresponding to the filling of $\lambda'(T)\setminus\nu(T)$ in the lower right corners. This should be clear on the following example of a path for the 
algebra ${\sf B}_{3,5}(\delta)$.

\begin{example}\label{exseq}
For the sequence $T$
\ytableausetup{boxsize=0.5em}
\begin{align*}
  \Big[\!\big(\!\varnothing,\varnothing\!\big),\big(\ydiagram{1},\varnothing\!\big),\big(\ydiagram{2},\varnothing\!\big),\big(\ydiagram{2,1},\varnothing\!\big),\big(\ydiagram{2,1},\ydiagram{1}\big),\big(\ydiagram{2,1},\ydiagram{2}\big),\big(\ydiagram{1,1},\ydiagram{2}\big),\big(\ydiagram{1},\ydiagram{2}\big),\big(\ydiagram{1},\ydiagram{2,1}\big)\Big]
\end{align*}
the corresponding triple diagram is $[(2,1),\,(1),\,(2,1) ]$ and the standard triple tableau $W_T$ is
\ytableausetup{boxsize=1.3em}
\begin{align*}
\Big[\,
\begin{ytableau}
^1_{\quad}  & ^2_{\,\,\,\,\, 6}\\
 ^3_{\,\,\,\,\, 7}
\end{ytableau},\,\,
\begin{ytableau}
4 & 5\\
8 
\end{ytableau}\,\Big] \ .
\end{align*}
\end{example}

The contents of boxes in the sets $\lambda'(T)$, $\lambda'(T)\setminus\nu(T)$ and $\lambda''(T)$ are calculated according to the formulas \eqref{c1},
\eqref{c3} and \eqref{c2} respectively. The content of the box occupied by the number $j$ in the standard triple diagram corresponding to a path $T$  will be denoted by $c_j(T)$.

\vskip .2cm
Let $T=(\Lambda^{(0)},\dots,\Lambda^{(r+s)})$, $\Lambda^{(r+s)}=\Lambda$, be a standard walled $\Lambda$-tableau. 
The vector $v_T$ is the common eigenvector for the Jucys-Murphy elements $x_j$, $j=1,\ldots ,r+s$, and the eigenvalues are precisely the contents,
\[x_jv_T=c_j(T)v_T\ ,\ j=1,\ldots ,r+s\ .\] 

We recall some standard facts valid in the situation when the branching rules are simple and the vectors $v_T$ are common eigenvectors of a set of elements generating a maximal 
commutative subalgebra. 

\vskip .2cm
Since the vectors $v_T$, $T\in\mathcal{T}_\Lambda$, form a basis in $V_\Lambda$, we have the complete set $\{ E_T\}$ of primitive idempotents  
in $\text{Mat}(V_\Lambda)$; the operator $E_T$ is the projector on the one-dimensional subspace $V_T$ along the subspace of codimension one spanned by the vectors $v_{T'}$, $T'\in\mathcal{T}_\Lambda\backslash \{ T\}$.  
The primitive idempotent $E_{T}$, corresponding to the vector $v_T$, satisfies 
\begin{equation}\label{xiet}
x_t\, E_{T}=E_{T}\, x_t=c_t(T) E_{T}\ , \ \ t=1,\dots,r+s\ .
\end{equation}

Consider the standard walled  tableau $U=(\Lambda^{(0)},\dots,\Lambda^{(r+s-1)})$; recall that we assume that $s>0$. 
Let $\alpha\in R\bigl(\nu(U)\bigr)\sqcup A\bigl(\lambda''(U)\bigr)$ be the box of $W_T$ occupied by the number $r+s$. By construction, we have
$$E_{T}=E_{U}\,
\frac{(x_{r+s}-a_1)\dots (x_{r+s}-a_\ell)}{(c_{r+s}(T)-a_1)\dots (c_{r+s}(T)-a_\ell)}\ ,$$ 
where $a_1,\dots,a_\ell$ are the contents of all boxes in $\biggl(R\bigl(\nu(U)\bigr)\sqcup A\bigl(\lambda''(U)\bigr)\biggr)\setminus\{\alpha\}$. 

\vskip .2cm
The elements $\{E_{T}\}$ for $T$ a standard  walled  $\Lambda$-tableau, $\Lambda\in\mathcal{D}_{r,s}$, form a complete set of pairwise orthogonal primitive idempotents for ${\sf B}_{r,s}(\delta)$. 

\vskip .2cm
We have also the following relation: 
\begin{equation}\label{jmform1}
E_{T}=E_{U} \frac{u-c_{r+s}}{u-x_{r+s}}\Big|^{}_{u=c_{r+s}},
\end{equation}
where $u$ is a complex variable. Indeed, the actions of the right and left hand sides of (\ref{jmform1}) on the vectors 
$v_W$, $W\in\mathcal{T}_\Lambda$, coincide. 

\begin{remark} The walled Brauer algebras resemble the symmetric groups in certain aspects (like the dimension, the length generating function for words \cite{BGO})
so it is natural to expect that the representation theory of the tower of the walled Brauer algebras  can be built in the frame of the inductive approach of 
Okounkov--Vershik to the representation theory of the symmetric groups \cite{OV}, see also \cite{IO} for the Hecke algebras, 
\cite{IO2} for the Birman--Murakami--Wenzl algebras and \cite{OP,OP2} for the cyclotomic algebras. 
\end{remark}

\subsection{Spectral parameters}
Denote, for $i\ne j$,
$$s_{i,j}(u):=1-\frac{s_{i,j}}{u}\ \  \text{ if }\ \varepsilon(i)+\varepsilon(j) \ \text{ is even}\ ,$$
$$d_{i,j}(u):=1-\frac{d_{i,j}}{u}\ \ \text{ if }\ \varepsilon(i)+\varepsilon(j) \ \text{ is odd}\ ,$$
where the function $\varepsilon$ is defined by (\ref{funepsi}). 
Let $w_{i,j}(u)$ be, depending on $\varepsilon(i)+\varepsilon(j) $, either $s_{i,j}(u)$ or $d_{i,j}(u)$. If the indices $i,j,k,l$ are pairwise distinct then 
$$ w_{i,j}(u)w_{k,l}(v)=w_{k,l}(v)w_{i,j}(u)\ .$$
We have
\begin{equation}
s_{i,j}(u)s_{i,j}(-u)=\frac{u^2-1}{u^2}\ ,\label{unitarm1}
\end{equation}
and
\begin{equation}\label{unid}
d_{i,j}(u)d_{i,j}(\delta -u)=1\ .
\end{equation}
The functions $s_{i,j}(u)$ satisfy the Yang-Baxter equation with the spectral parameter
\begin{align}\label{ybsp1}
s_{i,j}(u)s_{i,k}(u+v)s_{j,k}(v)&=s_{j,k}(v)s_{i,k}(u+v)s_{i,j}(u)\ ,
\end{align}
with pairwise distinct indices $i,j,k$. 

\vskip .2cm
Additionally, we have ($i\neq j\neq k\neq i$)
\begin{align}\label{ybsp3}
d_{j,i}(u)d_{k,i}(u-v)s_{j,k}(v)&=s_{j,k}(v)d_{k,i}(u-v)d_{j,i}(u)\ ,
\end{align}
and
\begin{align}\label{ybsp5}
d_{i,j}(u)s_{i,k}(\delta -u-v)d_{k,j}(v)=d_{k,j}(v)s_{i,k}(\delta -u-v)d_{i,j}(u)\ .
\end{align}

\begin{remark}
Equations (\ref{ybsp1}), (\ref{ybsp3}) and (\ref{ybsp5}) can be elegantly written in a uniform manner. Let
$$\tilde{w}_{i,j}(u):=\left\{\begin{array}{ll} s_{i,j}(u)&\text{ if }\ \varepsilon(i)+\varepsilon(j) \ \text{ is even}\ ,\\[.2em] 
d_{i,j}(\delta/2-u)&\text{ if }\ \varepsilon(i)+\varepsilon(j) \ \text{ is odd}\ . \end{array}\right.$$
Then 
$$\tilde{w}_{i,j}(u)\tilde{w}_{i,k}(u+v)\tilde{w}_{j,k}(v)=\tilde{w}_{j,k}(v)\tilde{w}_{i,k}(u+v)\tilde{w}_{i,j}(u)$$
whenever $i\neq j\neq k\neq i$.
\end{remark}

\section{Main results}\label{maosrere}
The main result of this paper is the construction of the complete set $\{ E_T\}$, see (\ref{xiet}), of pairwise orthogonal primitive idempotents for the walled Brauer algebra ${\sf B}_{r,s}(\delta)$. 

\subsection{Fusion procedure}\label{fuprose}
The original fusion procedure for the Brauer algebra was given in \cite{IM}. In this section we formulate its analogue for the walled Brauer algebra. The proof is in Section
\ref{promath}.

\vskip .2cm
In what follows we let $$n=r+s\ .$$

Consider the rational function, in variables $u_1,\dots,u_{n}$, with values in the walled Brauer algebra ${\sf B}_{r,s}(\delta)$:
\begin{equation}\label{psimain}
\Psi_{r,s}:=\mathfrak{D}_{r,s}\, \mathfrak{S}_r\, \mathring{\mathfrak{S}}_s\ ,
\end{equation}
where
$$\mathfrak{D}_{r,s}:=\prod_{\substack{ 1\leqslant i\leqslant r \\ r+1\leqslant j\leqslant n}} d_{i,j}(u_i+u_j)$$
and
\begin{equation}\label{sbars}\mathfrak{S}_r:=\prod_{\substack{ 1\leqslant i<j \leqslant r}}
s_{i,j}(u_i-u_j)\ \ ,\ \ \mathring{\mathfrak{S}}_s:=\prod_{\substack{ r+1\leqslant i<j\leqslant n}} {s}_{i,j}(u_i-u_j)\ .\end{equation}
The products in the definitions of $\mathfrak{D}_{r,s}$, $\mathfrak{S}_r$ and $\mathring{\mathfrak{S}}_s$ are calculated in the lexicographical order on the pairs $(i,j)$ (that is, $(i_1,j_1)$ precedes $(i_2,j_2)$ if $i_1 < i_2$ or $i_1 = i_2$ and $j_1 < j_2$). 

\vskip .2cm
Let $T=(\Lambda^{(0)},\dots,\Lambda^{(n)})$, $\Lambda^{(n)}=\Lambda\in\mathcal{D}_{r,s}$, be a standard walled $\Lambda$-tableau
describing a path in the Bratteli diagram for the walled Brauer algebra ${\sf B}_{r,s}(\delta)$. 
We define the rational function in the variables $u_1,\dots,u_{n}$:
\begin{equation}\label{predf}
z_T := \prod_{i=1}^{n}\ \frac{u_i-c_i}{u_i-\delta\,\varepsilon (i)}\ \times
\prod_{\substack{   1\leqslant j<i\leqslant r \\ \text{or}\\   r<j<i\leqslant n  }}\frac{(u_i-u_j)^2}{(u_i-u_j)^2-1} \ ,\end{equation}
where the function $\varepsilon$ is defined by (\ref{funepsi}), and for brevity we denoted $c_i=c_i(T)$, $i=1,\dots,n$.

\vskip .2cm
Set
\begin{equation}\label{fufufu} \Psi_T(u_1,\dots,u_{n}):=z_T  \cdot  \Psi_{r,s} \ .\end{equation}
\begin{theorem}\label{thm:fus}
The primitive idempotent $E_T$, corresponding to the standard walled $\Lambda$-tableau $T$, is found by
the consecutive evaluations
$$E_T=\Psi_T(u_1,\dots,u_{n})\big|_{u_1=c_1}\big|_{u_2=c_2}\dots \big|_{u_{n}=c_{n}}\ .$$
\end{theorem}

\vskip .2cm
\begin{example}\label{exidem}
For $r=s=2$, let $T$ be the standard walled  tableau corresponding to the contents sequence $(0,-1,1,0)$, see Figure \ref{fig:2}. We have
$$\begin{array}{rcl}
 \Psi_{2,2}(0,u_2,u_3,u_4)&\!\!\!=\!\!\!&\displaystyle{ \left(\!1-\frac{d_{1,3}}{u_3}\right) \left(\!1-\frac{d_{1,4}}{u_4}\right) 
\left(\!1-\frac{d}{u_2+u_3}\right) \left(\!1-\frac{d_{2,4}}{u_2+u_4}\right) }\\[2em] 
&&\times\ \displaystyle{  \left(\!1+\frac{s_1}{u_2}\right) \left(\!1-\frac{s_3}{u_3-u_4}\right)  }\ .\end{array}$$
This expression has singularities at $u_3=-u_2$ or $u_4=0$. However in the process of the consecutive evaluations of the product of $\Psi_{2,2}(u_1,u_2,u_3,u_4)$ with the 
prefactor $z_T$ the singularities cancel and we find
$$E_T=\frac{1}{2\delta (\delta-1)}\, (1-s_1)\cdot ds_1 s_3 d\cdot  (1-s_1)\ .$$
\end{example}

\begin{remark}
The function
\begin{equation}\label{nafu}
\prod_{1\leqslant i< j\leqslant n} d_{i,j}(u_i+u_j)\ \times \prod_{1\leqslant i< j\leqslant n} s_{i,j}(u_i-u_j)
\end{equation}
for the Brauer algebra ${\sf B}_{n}(\delta)$ was suggested in \cite{Na2}. Note that the function $\Psi_{r,s}$, defined in (\ref{psimain}), can be obtained by dropping, in the expression
(\ref{nafu}), factors corresponding to the diagrams which do not exist in the walled Brauer algebra. Thus the function 
$\Psi_{r,s}$ makes sense in the Brauer algebra ${\sf B}_{n}(\delta)$ and the consecutive evaluations of the product 
$z_T  \cdot  \Psi_{r,s}$ give rise to certain idempotents of the Brauer algebra
${\sf B}_{n}(\delta)$. It would be interesting to understand the representation-theoretic/combinatorial meaning of these idempotents.
\end{remark}

\subsubsection{Reformulation of Theorem \ref{thm:fus}} \label{reform3}
As we already stressed, we are interested in the fusion procedure only after the wall crossing so we shall accordingly change the notation. 
A standard walled $\Lambda$-tableau $T=(\Lambda^{(0)},\dots,\Lambda^{(r+s)})$, $\Lambda^{(r+s)}=\Lambda$, 
will be denoted by $T=(T_r,\Lambda^{(r+1)},\dots,\Lambda^{(r+s)})$ where $T_r=(\Lambda^{(0)},\dots,\Lambda^{(r)})$
is the standard Young tableau of shape $\lambda'(T)$. We fix the tableau $T$ till the end of Section and set $c_i=c_i(T)$. 

\vskip .2cm
For $j$ such that $r+s\geqslant j>r$ let
$$\mathfrak{d}_j^{\downarrow}:=d_{r,j}(u_{r}+u_j)d_{r-1,j}(u_{r-1}+u_j)\ldots d_{1,j}(u_{1}+u_j)\ $$
and
$$\mathring{\mathfrak{d}}_j^{\downarrow}:=\mathfrak{d}_j^{\downarrow}\big|_{u_{1}=c_{1}}\dots\big|_{u_{r}=c_{r}}=d_{r,j}(c_{r}+u_j)d_{r-1,j}(c_{r-1}+u_j)\ldots 
d_{1,j}(c_{1}+u_j)\ .$$
The symbol $\mathring{}$ (it appeared already in (\ref{psimain})) over a letter signifies that we are dealing with 
a rational function which depends on the variables $u_{r+1},\dots,u_n$ only.

\vskip .2cm
With the help of the equalities (\ref{ybsp3}), see Section \ref{promath}, one finds
$$\mathfrak{D}_{r,s}\, \mathfrak{S}_r=\mathfrak{S}_r \, \mathfrak{d}_{r+1}^{\downarrow}\mathfrak{d}_{r+2}^{\downarrow}\ldots \mathfrak{d}_n^{\downarrow}\ .$$

The fusion procedure of \cite {Mo} for the symmetric group says that the primitive idempotent $E_{T_r}$ corresponding to the standard tableau 
$T_r$ of the symmetric group
$\mathbb{S}_r$ is obtained by the consecutive evaluations
$$E_{T_r}=\left(\prod_{i=1}^{r}\ \frac{u_i-c_i}{u_i}\ \times
\prod_{ 1\leqslant j<i\leqslant r}\frac{(u_i-u_j)^2}{(u_i-u_j)^2-1}\cdot \mathfrak{S}_r\right)\bigg|_{u_{1}=c_{1}}\dots\bigg|_{u_{r}=c_{r}} \ .$$
The part of the prefactor $z_T$, see (\ref{predf}), which corresponds to the after-wall tail $(\ldots,\Lambda^{(r+1)},\dots,\Lambda^{(r+s)})$ of the tableau $T$, is the rational function in the variables $u_{r+1},\dots,u_{n}$:
$$\mathring{z}_T := \prod_{i=1}^{n}\ \frac{u_i-c_i}{u_i-\delta}\ \times
\prod_{  r<j<i\leqslant n }\frac{(u_i-u_j)^2}{(u_i-u_j)^2-1} \ .$$
Let now
$$\mathring{\Psi}_{n;T_r}(u_{r+1},\dots,u_{n}):=E_{T_r}\mathring{\mathfrak{d}}_{r+1}^{\downarrow}\mathring{\mathfrak{d}}_{r+2}^{\downarrow}\ldots \mathring{\mathfrak{d}}_n^{\downarrow}\,  \mathring{\mathfrak{S}}_s$$
and
$$\mathring{\Psi}_{T}(u_{r+1},\dots,u_{n}):=\mathring{z}_{T}  \cdot  \mathring{\Psi}_{n;T_r}(u_{r+1},\dots,u_{n})\ .$$
We reformulate Theorem \ref{thm:fus} in the following way: {\it the idempotent $E_T$ is found by the consecutive evaluations}
$$E_T=\mathring{\Psi}_{T}(u_{r+1},\dots,u_{n})\big|_{u_{r+1}=c_{r+1}}\dots\big|_{u_{n}=c_{n}} \ .$$

\subsection{Exponents}
The function $\Psi_T$ whose consecutive evaluations give the primitive idempotents, has two parts: the function $z_T$, defined in (\ref{predf}) and the function $\Psi_{r,s}$,  
defined in (\ref{psimain}). The consecutive evaluations of the function $\Psi_{r,s}$ itself may lead to a zero or infinite result. The presence of the prefactor   
$z_T$ cures this undesirable effect. One can locate the minimal collection of the factors in $z_T$ which are needed to provide a finite result. In \cite{IM}, Isaev and Molev
introduced the `exponents' which control such factors for the Brauer algebra. In this Section we adapt the exponents to the situation of the walled Brauer algebra.  
The exponents for the walled Brauer algebra have simpler form than for the full Brauer algebra. 

\vskip .2cm
We need a general definition. Fix a Young diagram $\gamma$. We denote by $g_{\gamma}(k)$ the number of cells on the $k$-th diagonal of $\gamma$. We shall define a function $\vartheta_{\gamma}\colon \mathbb{Z}\to\{ -1,0,1\}$ by the following rules.
\begin{itemize}
\item If, adding to $\gamma$ the $\left( g_{\gamma}(k)+1\right)$-st cell to the $k$-th diagonal, we obtain a Young diagram then we set  $\vartheta_{\gamma}(k)=-1$.
\item If, removing from $\gamma$ the $g_{\gamma}(k)$-st cell on the $k$-th diagonal, we obtain a Young diagram then we set  $\vartheta_{\gamma}(k)=1$.
\item Otherwise, $\vartheta_{\gamma}(k)=0$. 
\end{itemize}

Let $T=(\Lambda^{(0)},\dots,\Lambda^{(n)})$ be a standard walled tableau. We will define the exponents $p_t(T)$, $t=1,\dots,n$. 
If $\Lambda^{(t)}$ is obtained from $\Lambda^{(t-1)}$ by adding a box to the first or second diagram in the bipartition then we let  $p_t(T)=0$. Assume now that 
$\Lambda^{(t)}$ is obtained from $\Lambda^{(t-1)}$ by removing a box $\alpha=(i,j)$ from the first diagram in the bipartition. In particular, at this moment the wall is 
already crossed so the diagram $\lambda'(T)$ is defined. The exponents of $T$ are defined by  
\begin{equation}\label{deexpn}p_t(T):=\vartheta_{\lambda'(T)}(i-j)\ .\end{equation}
We introduce now the rational function
$${\sf ze}_T:=\prod_{i=r+1}^n (u_i-c_i)^{p_i(T)}\ .$$
\begin{proposition}\label{proex} In the notation of Section \ref{fuprose}, the consecutive evaluations of the function 
$${\sf ze}_T\cdot \Psi_{r,s}$$
are finite.

\vskip .2cm\noindent
If $p_i(T)\neq 0$ then the factor $(u_i-c_i)^{p_i(T)}$ is present in $z_T$. The consecutive evaluations of the product $z_T/{\sf ze}_T$ (of all other factors) are finite.  
\end{proposition}
Note that before crossing the wall the exponents are equal to zero. This follows from the fusion procedure for the symmetric group \cite{Mo}, see also \cite{IM}. 

\vskip .2cm
The proof of Proposition \ref{proex} is in Section \ref{secproex}.

\subsection{Second fusion procedure}
Our second fusion procedure resembles the fusion procedure of \cite{IMO} for the Brauer algebra. 
\paragraph{Modified baxterized elements.} Let $\hh$ be an indeterminate.
It will be convenient to use also the following modified functions:
\begin{align*}
s'_{i,j}(u;\hh)&:=1+\frac{s_{i,j}}{u-\hh}\ \  \text{ if }\ \varepsilon(i)+\varepsilon(j) \ \text{ is even}\ ,\\
d'_{i,j}(u;\hh)&:=1+\frac{d_{i,j}}{u+\hh-\delta}\ \  \text{ if }\ \varepsilon(i)+\varepsilon(j) \ \text{ is odd}\ .
\end{align*}
In virtue of the equalities (\ref{ybsp1}), we have
$$s'_{i,j}(u;\hh)s'_{i,k}(u-v;\hh)s_{j,k}(v)=s_{j,k}(v)s'_{i,k}(u-v;\hh)s'_{i,j}(u;\hh)\ ,$$
in virtue of the equalities  (\ref{ybsp3}),
$$d'_{j,i}(v;\hh)d'_{k,i}(u+v;\hh)s_{j,k}(u)=s_{j,k}(u)d'_{k,i}(u+v;\hh)d'_{j,i}(v;\hh)\ ,$$
and, in virtue of the equalities  (\ref{ybsp5}),
$$d_{i,j}(u)s'_{i,k}(u-v;\hh)d'_{k,j}(v;\hh)=d'_{k,j}(v;\hh)s'_{i,k}(u-v;\hh)d_{i,j}(u)\ .$$
In the sequel we omit the symbol $\hh$ in the notation for brevity.

\paragraph{Second fusion procedure.} We shall formulate the results in the notation of Subsection \ref{reform3}.

\vskip .2cm
We define several more rational functions, in variables $u_{r+1},\dots,u_{n}$, with values in the algebra ${\sf B}_{r,s}(\delta)$.
First, for $j$ such that $r+s\geqslant j>r$ let
$$\mathring{\mathfrak{d}}_j^{'\uparrow}:=d_{1,j}'(c_{1}-u_j)d_{2,j}'(c_{2}-u_j)\ldots d_{r,j}'(c_{r}-u_j)\ .$$
Next, let
$$\mathring{\mathfrak{A}}_{r,s}:=\mathring{\mathfrak{d}}_{r+1}^{\downarrow}\mathring{\mathfrak{d}}_{r+2}^{\downarrow}\ldots \mathring{\mathfrak{d}}_n^{\downarrow}\ ,\ 
\mathring{\mathfrak{A}}_{r,s}':=\mathring{\mathfrak{d}}_{r+1}^{'\uparrow}\mathring{\mathfrak{d}}_{r+2}^{'\uparrow}\ldots \mathring{\mathfrak{d}}_n^{'\uparrow}\ .$$
Finally, let
$$\begin{array}{c}
\mathring{\Psi}_{n;T_r,h}(u_{r+1},\dots,u_{n}):=E_{T_r}\mathring{\mathfrak{A}}_{r,s}'\, \mathring{\mathfrak{S}}_s'\, \mathring{\mathfrak{A}}_{r,s}\, 
\mathring{\mathfrak{S}}_s\ ,\\[.6em]
\mathring{\widetilde\Psi}_{n;T_r,h}(u_{r+1},\dots,u_{n}):=E_{T_r}\mathring{\mathfrak{A}}_{r,s}\, \mathring{\mathfrak{S}}_s'\, \mathring{\mathfrak{A}}_{r,s}'\, \mathring{\mathfrak{S}}_s\ ,\end{array}$$
where $\mathring{\mathfrak{S}}_s$ is defined in (\ref{sbars}) and
$$\mathring{\mathfrak{S}}_s' :=\prod_{\substack{ r+1\leqslant i<j\leqslant r+s}} {s}_{i,j}'(u_i+u_j)\ .$$
The prefactor $\mathring{z}_T$ we replace with the following rational function in the variables $u_{r+1},\dots,u_{n}$:
\begin{equation}\label{nopref}\mathring{z}_{T;h}:=\prod_{i=r+1}^{n}\frac{(u_i-c_i)(u_i-h+\delta)}{(u_i-\delta )(u_i+c_i-h)}\, \times
\prod_{r<j<i\leqslant n  }\frac{(u_i-u_j)^2}{(u_i-u_j)^2-1}\ .\end{equation}
Set
$$\begin{array}{c} 
\mathring{\Psi}_{T;h}(u_{r+1},\dots,u_{n}):=\mathring{z}_{T;h}  \cdot  \mathring{\Psi}_{n;T_r,h}(u_{r+1},\dots,u_{n}) \ ,\\[.6em] 
\mathring{\widetilde{\Psi}}_{T;h}(u_{r+1},\dots,u_{n}):=\mathring{z}_{T;h} \cdot \mathring{\widetilde{\Psi}}_{n;T_r,h}(u_{r+1},\dots,u_{n}) \ .
\end{array}$$

\begin{proposition}\label{prop:fusb2}
The primitive idempotent $E_T$, corresponding to the standard walled $\Lambda$-tableau $T$, is found by any of 
the consecutive evaluations
\begin{equation}\label{fubtve1} 
E_T=\mathring{\Psi}_{T;h}(u_{r+1},\dots,u_{n})\big|_{u_{r+1}=c_{r+1}}\dots\big|_{u_{n}=c_{n}} \end{equation}
or
\begin{equation}\label{fubtve2}
E_T=\mathring{\widetilde{\Psi}}_{T;h}(u_{r+1},\dots,u_{n})\big|_{u_{r+1}=c_{r+1}}\dots\big|_{u_{n}=c_{n}}\ .\end{equation}
\end{proposition}

\vskip .2cm
\begin{remark} The fusion procedure of Theorem \ref{thm:fus} is the limit, as $h$ tends to infinity, of the second fusion procedure, given in Proposition
\ref{prop:fusb2}.\end{remark}

\section{Proofs}
\subsection{Proof of Theorem \ref{thm:fus}}\label{promath}
We repeat that we assume that $s>0$ because before crossing the wall our formulas reproduce the formulas for the symmetric groups from \cite{Mo}. 

\vskip .2cm
We shall often write $d_{i,j}(u,v)$ and $s_{i,j}(u,v)$ instead of $d_{i,j}(u+v)$ and 
$s_{i,j}(u-v)$.  

\vskip .2cm
We rewrite the function $\Psi_{r,s}$, defined by \eqref{psimain}, in the form adapted to the consecutive evaluations. Let 
\begin{equation}\label{depsi}\psi_n:=d_{r,n}(u_r,u_n)\dots d_{1,n}(u_1,u_n)\cdot {s}_{r+1,n}(u_{r+1},u_n)\dots {s}_{n-1,n}(u_{n-1},u_n)\ .\end{equation}
\begin{lemma}\label{psiind} We have 
$$\Psi_{r,s}=\Psi_{r,s-1}\cdot \psi_n\ .$$
\end{lemma}
\begin{proof} Clearly,  $\mathfrak{D}_{r,s}=\mathfrak{D}_{r,s-1}\cdot d_{1,n}(u_1,u_n)\dots d_{r,n}(u_r,u_n)$.
The Yang--Baxter equations (\ref{ybsp3}) imply the identity
$$d_{1,n}(u_1,u_n)\dots d_{r,n}(u_r,u_n)\cdot \mathfrak{S}_r=\mathfrak{S}_r\cdot d_{r,n}(u_r,u_n)\dots d_{1,n}(u_1,u_n)\ .$$
The well-known equality
$\mathring{\mathfrak{S}}_s=\mathring{\mathfrak{S}}_{s-1}\cdot {s}_{r+1,n}(u_{r+1},u_n)\dots {s}_{n-1,n}(u_{n-1},u_n)$
and the commutativity relation 
$$d_{r,n}(u_r,u_n)\dots d_{1,n}(u_1,u_n)\cdot \mathring{\mathfrak{S}}_{s-1}=\mathring{\mathfrak{S}}_{s-1}\cdot d_{r,n}(u_r,u_n)\dots d_{1,n}(u_1,u_n)\ ,$$
complete the proof.\end{proof}
 
\vskip .2cm  
We first analyze what happens when  we cross the wall.

\begin{lemma}\label{wacro} Let $U$ be a standard walled tableau for the algebra ${\sf B}_{r,0}(\delta)$, that is, for the symmetric group $\mathbb{S}_r$.
The following identity holds in the walled Brauer algebra ${\sf B}_{r,1}(\delta)$:
\begin{equation}\label{focrowa}E_U\cdot d_{1,r+1}(w-c_1) \ldots d_{r,r+1}(w-c_r)=\frac{w-\delta+x_{r+1}}{w}\cdot E_U\ ,\end{equation}
where $c_i=c_i(U)$, $i=1,\dots,r$.
\end{lemma}
\begin{proof} The proof is by induction in $r$.  The induction base, for the algebra ${\sf B}_{1,1}(\delta)$, is straightforward.

\vskip .2cm
Let $W\nearrow U$. By the recursive construction of the primitive idempotents for the symmetric group in \cite{Mo} (which is exactly the `before the wall' part 
of our formula), we have  
$$E_U=b(v)\, E_W\cdot s_{1,r}(c_1,v)\ldots s_{r-1,r}(c_{r-1},v)\vert_{v=c_r}$$
with some rational function $b(v)$; its precise expression is not important at the moment. Let $\xi=s_{1,r}(c_1,v)\ldots s_{r-1,r}(c_{r-1},v)$. 
The Yang--Baxter equations (\ref{ybsp3}) imply that
$$\begin{array}{l}
\xi \cdot d_{1,r+1}(w-c_1) \ldots d_{r-1,r+1}(w-c_{r-1})\cdot d_{r,r+1}(w-v)\\[.6em]
\hspace{1cm}= d_{r,r+1}(w-v)\cdot d_{1,r+1}(w-c_1) \ldots d_{r-1,r+1}(w-c_{r-1})\cdot \xi\ .\end{array}$$
Since 
$$E_W\,d_{r,r+1}(w-v)=d_{r,r+1}(w-v)\,E_W\ ,$$
we can write the left hand side of (\ref{focrowa}) in the form
$$b(v)d_{r,r+1}(w-v)\, E_W\cdot d_{1,r+1}(w-c_1) \ldots d_{r-1,r+1}(w-c_{r-1})\cdot\xi\vert_{v=c_r}\ .$$
The diagrams with the vertical line connecting the $r$-th upper and lower points linearly span the subalgebra $\mathcal{A}$ in ${\sf B}_{r,1}(\delta)$, isomorphic to
${\sf B}_{r-1,1}(\delta)$. Thus we can use the induction hypothesis and write
$$E_W\cdot d_{1,r+1}(w-c_1) \ldots d_{r-1,r+1}(w-c_{r-1})\!=\!\frac{w-d_{1,r+1}-\dots -d_{r-1,r+1}}{w}\cdot E_W\ .$$
Now,
$$d_{r,r+1}(w-v)\cdot\frac{w-d_{1,r+1}-\dots -d_{r-1,r+1}}{w}$$
$$=\frac{w-d_{1,r+1}-\dots -d_{r-1,r+1}}{w}
-\frac{d_{r,r+1}\cdot (w-d_{1,r+1}-\dots -d_{r-1,r+1})}{w(w-v)}$$
$$=\frac{w-d_{1,r+1}-\dots -d_{r-1,r+1}}{w}-\frac{d_{r,r+1}\cdot (w-s_{1,r}-\dots -s_{r-1,r})}{w(w-v)}$$
$$=\frac{w-d_{1,r+1}-\dots -d_{r-1,r+1}}{w}-\frac{d_{r,r+1}\cdot (w-x_r)}{w(w-v)}$$
$$=\frac{w-\delta+x_{r+1}}{w}-\frac{d_{r,r+1}\cdot (v-x_r)}{w(w-v)}\ .$$
Here in the second equality we used the formula $d_{r,r+1}d_{j,r+1}=d_{r,r+1}s_{j,r}$. 

\vskip .2cm
We have
$$b(v)\,\frac{d_{r,r+1}\cdot (v-x_r)}{w(v-w)}\cdot E_W\cdot\xi\vert_{v=c_r}=\frac{d_{r,r+1}\cdot (v-x_r)}{w(v-w)}\cdot E_U\vert_{v=c_r}\ ,$$
so we are done since $x_rE_U=c_rE_U$.
\end{proof}
 
\begin{lemma}\label{wacro2} Let $U$ be a standard walled tableau for the algebra ${\sf B}_{r,s-1}(\delta)$ with $s>0$.
Then
\begin{equation}\label{focrowa1}E_U\cdot\zeta_n=\frac{u-x_{n}}{u-\delta}\cdot E_U\ ,\end{equation}
where $\zeta_n$ is the following rational function in the variable $u$:
$$\zeta_n={s}_{n-1,n}(u,c_{n-1})\dots {s}_{r+1,n}(u,c_{r+1})\cdot d_{1,n}(\delta-u,-c_1) \ldots d_{r,n}(\delta-u,-c_r)$$
with $c_i=c_i(U)$, $i=1,\dots,n-1$.
\end{lemma}
\begin{proof} We employ the induction on $s$. The induction base, for $s=1$ is the formula (\ref{focrowa}). Let $W\nearrow U$. We have $E_U=E_U E_W$. Thus 
$$E_U\cdot \zeta_n=E_U E_W\cdot {s}_{n-1,n}(u,c_{n-1})\zeta'_n=E_U\cdot  {s}_{n-1,n}(u,c_{n-1})\cdot  E_W\cdot \zeta'_n\ ,$$ 
where
$$\zeta'_n :={s}_{n-2,n}(u,c_{n-2})\dots{s}_{r+1,n}(u,c_{r+1})\cdot d_{1,n}(\delta-u,-c_1) \ldots d_{r,n}(\delta-u,-c_r)\ .$$
The diagrams with the vertical line connecting the $(n-1)$-st upper and lower points linearly span the subalgebra $\mathcal{A}'$ in ${\sf B}_{r,s}(\delta)$, isomorphic to
${\sf B}_{r,s-1}(\delta)$.
By the induction assumption, 
$$E_W\cdot \zeta'_n=E_W\rho_n\ ,$$
where
$$\rho_n :=\frac{u-(\delta-d_{1,n}-\dots -d_{r,n}+{s}_{r+1,n}+\dots +{s}_{n-2,n}) }{u-\delta}\ .$$
So,
$$E_U\cdot\zeta_n=E_U\cdot {s}_{n-1,n}(u,c_{n-1})\cdot E_W\cdot\rho_n=E_UE_W\cdot {s}_{n-1,n}(u,c_{n-1}) \rho_n$$
$$=E_U\cdot {s}_{n-1,n}(u,c_{n-1}) \rho_n\ .$$
Now,
\begin{equation}\label{focrowa2}{s}_{n-1,n}(u,c_{n-1}) \rho_n=\left(1-\frac{ {s}_{n-1,n}}{u-c_{n-1}}\right)  \rho_n
=  \rho_n-\frac{ {s}_{n-1,n}}{u-c_{n-1}}  \rho_n\ .\end{equation}
Since
$$\rho_n=\frac{u-x_n+{s}_{n-1,n} }{u-\delta}\ \text{ and }\  {s}_{n-1,n}\rho_n=\frac{u-x_{n-1} }{u-\delta} {s}_{n-1,n}\ ,$$
we rewrite the right hand side of the formula (\ref{focrowa2}) in the form
$$\frac{u-x_n+{s}_{n-1,n} }{u-\delta}-\frac{1}{u-c_{n-1}}\frac{u-x_{n-1} }{u-\delta} {s}_{n-1,n}$$
$$=\frac{u-x_n }{u-\delta}+\frac{1}{u-\delta}\left( 1-\frac{u-x_{n-1} }{u-c_{n-1}  }\right) {s}_{n-1,n}\ ,$$
which completes the proof since $E_U x_{n-1}=c_{n-1} E_U$. \end{proof}

\vskip .2cm
We define the following rational function in the variable $u=u_n$:
$$\psi_n^U:=\psi_n\vert_{u_1=c_1,\dots ,u_{n-1}=c_{n-1}}$$
where $\psi_n$ is defined in (\ref{depsi}).

\begin{lemma}\label{lem:jmsi}
The following identity holds:
\begin{equation}\label{jmid}
E_U\cdot \psi_n^U=\displaystyle{ \frac{u-\delta}{u-c_n} \prod_{i=r+1}^{n-1}\Big(1-\frac{1}{(u-c_i)^2}\Big) E_U \frac{u-c_{n}}{u-x_{n}} }\ .
\end{equation}
\end{lemma}

\begin{proof} The formula (\ref{jmid}) is obtained from the formula (\ref{focrowa1}) by using the equalities (\ref{unitarm1}) and (\ref{unid}).
\end{proof}

\vskip .2cm
\noindent{\bf Proof} of Theorem \ref{thm:fus}. Clearly,
\begin{equation}\label{jztu}z_T=z_U z_U^T\ \ \ \text{where}\ \ \ 
z_U^T= \frac{u_n-c_n}{u_n-\delta}\prod_{i=r+1}^{n-1}\,\frac{(u_i-u_n)^2}{(u_i-u_n)^2-1}\ .\end{equation}
Thus,
$$\Psi_T(u_1,\dots,u_{r+s})\big|_{u_1=c_1}\big|_{u_2=c_2}\dots
\big|_{u_{n}=c_{n}}=
\left( z_U^T\big|_{u_1=c_1,\dots ,u_{n-1}=c_{n-1}}  E_U\psi_n^U\right)\big|_{u_n=c_n}$$
$$=E_U \frac{u-c_{n}}{u-x_{n}}\big|_{u_n=c_n}=E_T$$
by (\ref{jmform1}).\hspace{.4cm}\myblacksquare

\subsection{Proof of Proposition \ref{proex}}\label{secproex}
As for the fusion procedure itself, it is sufficient to analyze only the last evaluation $\ldots\big|_{u_n=c_n}$. Let $U\nearrow T$. Since the last evaluation of the product 
$\Psi_T(u_1,\dots,u_{n})=z_T  \cdot  \Psi_{r,s}$ is finite by Theorem \ref{thm:fus}, we only need to locate the dangerous factors (zeros and poles) in the function 
\begin{equation}\label{dafainpre}z_U^T\big|_{u_1=c_1,\dots ,u_{n-1}=c_{n-1} }= \frac{u_n-c_n}{u_n-\delta}\prod_{i=r+1}^{n-1}\,\frac{(u_n-c_i)^2}{(u_n-c_i)^2-1}\ .\end{equation}
We shall rewrite in  another form the product
$$\varpi_U(u):=\prod_{i=r+1}^{n-1}\, \frac{(u-c_i)^2}{(u-c_i)^2-1}\ .$$
Note that $\varpi_U(u)$ depends only on the shape of the diagrams $\lambda'(U)$, $\nu(U)$ and $\lambda''(U)$. Indeed, the product 
$\prod_{i=r+1}^{n-1} (u-c_i)$ is the same as the product $\prod (u-c(\alpha))$ over all cells $\alpha$ of  $\lambda'(U)\setminus\nu(U)$ and $\lambda''(U)$. The contents are constant along the diagonals, so each factor will appear with the multiplicity equal to the length of the corresponding diagonal. 
Taking into account the definitions (\ref{c3}) and (\ref{c3}) of the contents, we find
$$\prod_{i=r+1}^{n-1} (u-c_i)=\prod_{k\in\mathbb{Z}}\, (u+k)^{g_{\lambda'(U)\setminus\nu(U)}(k)}\times 
\prod_{k\in\mathbb{Z}}\, (u-k-\delta)^{g_{\lambda''(U)}(k)}\ .$$
The denominator of $\varpi_U(u)$ contains two products, with factors $( u-c_i-1)$ and $( u-c_i+1)$. 
Shifting correspondingly the index in each product, we find that the whole expression $\varpi_U(u)$ is equal to
$$\varpi_U(u)=\varpi_U'(u)\ \varpi_U''(u)\ ,$$
where
$$\varpi_U'(u):=\prod_{k\in\mathbb{Z}}\, (u+k)^{\Delta_{\lambda'(U)\setminus\nu(U)}(k)}\ \ ,\ \  
\varpi_U''(u):= \prod_{k\in\mathbb{Z}}\, (u-k-\delta)^{\Delta_{\lambda''(U)}(k)}\ ,$$
and we denoted, for a diagram $\gamma$, the Laplacian of the function $g_{\gamma}$ with the minus  sign,  
$$\Delta_{\gamma}(k):=2g_{\gamma}(k)-g_{\gamma}(k+1)-g_{\gamma}(k-1)\ .$$
We analyze the function
$$\frac{u_n-c_n}{u_n-\delta}\, \varpi_U'(u_n)\ \varpi_U''(u_n)\ ,$$
depending  on what happens at the $n$-th step. 
\begin{itemize}
\item If $T$ is obtained from $U$ by adding a cell to the second diagram
on a diagonal $\underline{k}$ then $c_n=\underline{k}+\delta$. The function $\varpi_U'(u_n)$ is regular at $u_n=c_n$
and the direct inspection shows that the function $\frac{u_n-c_n}{u_n-\delta}\,\varpi_U''(u_n)$ is regular at $u_n=c_n$
as well (similarly to the fusion for the symmetric group). Thus the last evaluation of the function $\Psi_{r,s}$ is finite, 
which corresponds to our statement that $p_n(T)=0$ in this case.  

\item If $T$ is obtained from $U$ by removing a cell from the first  diagram on a diagonal $\underline{k}$ then $c_n=-\underline{k}$. The function $\frac{1}{u_n-\delta}\varpi_U''(u_n)$ is regular at $u_n=c_n$. However the function
$(u_n+\underline{k})\,\varpi_U'(u_n)$ may have singularities at $u_n=c_n$. The Laplacian $\Delta$ can take several 
values depending on the location of the diagonal $\underline{k}$. Different types of the location are illustrated on 
Figures \ref{nu-lambda-a}, \ref{nu-lambda-r} and \ref{nu-lambda-0} (the diagonal is shown in red), and in each case the 
value $\Delta_{\lambda'(U)\setminus\nu(U)}(\underline{k})$ is given - it can be directly read off the corresponding figure. 
There is one more case when the diagonal runs into the horizontal line of the border of $\lambda'$; then, similarly to the 
case of the vertical line of the border $\lambda'$, as on Figure \ref{nu-lambda-0}, we have 
$\Delta_{\lambda'(U)\setminus\nu(U)}(\underline{k})=-1$. In each case it is straightforward to see that the needed
value $p_n(T)$ of the exponent is indeed given by (\ref{deexpn}).
\end{itemize}

\begin{figure}[H]
\centering
\includegraphics[scale=0.24]{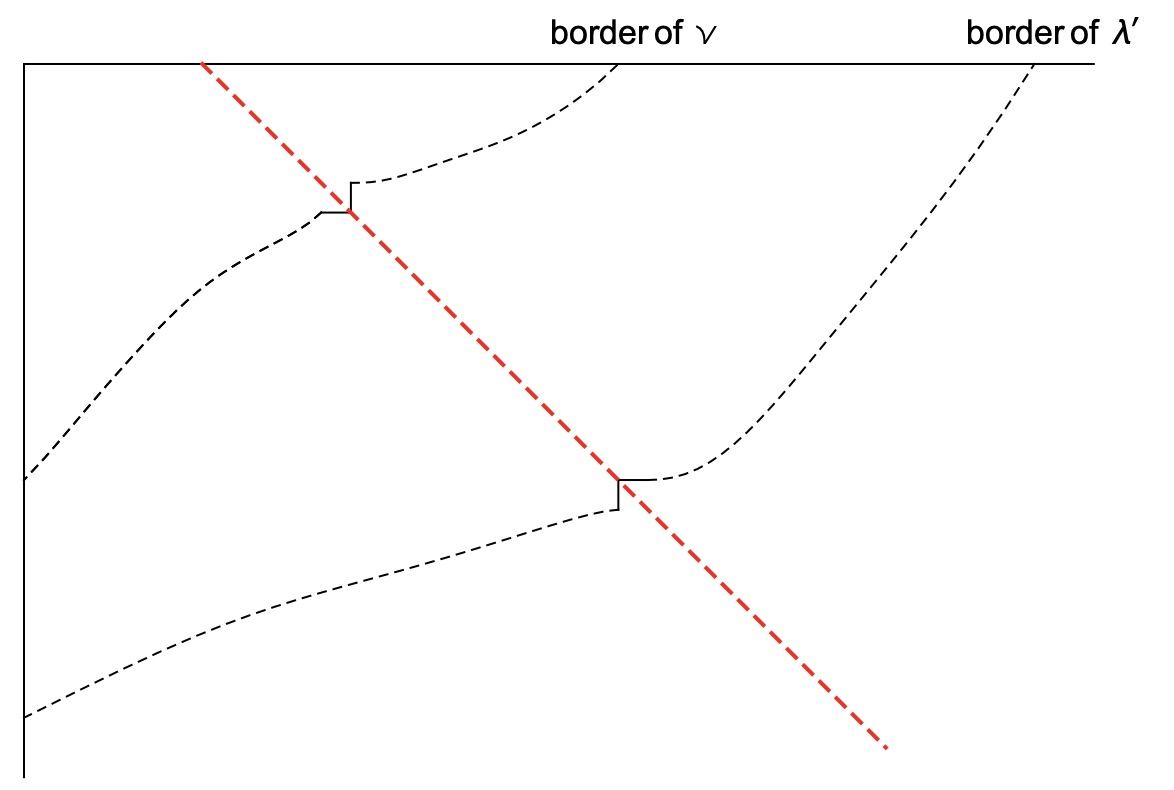}
\caption{$\Delta_{\lambda'(U)\setminus\nu(U)}(\underline{k})=-2$}
\label{nu-lambda-a}
\end{figure}

\vskip 1.8cm
\begin{figure}[H]
\centering
\includegraphics[scale=0.24]{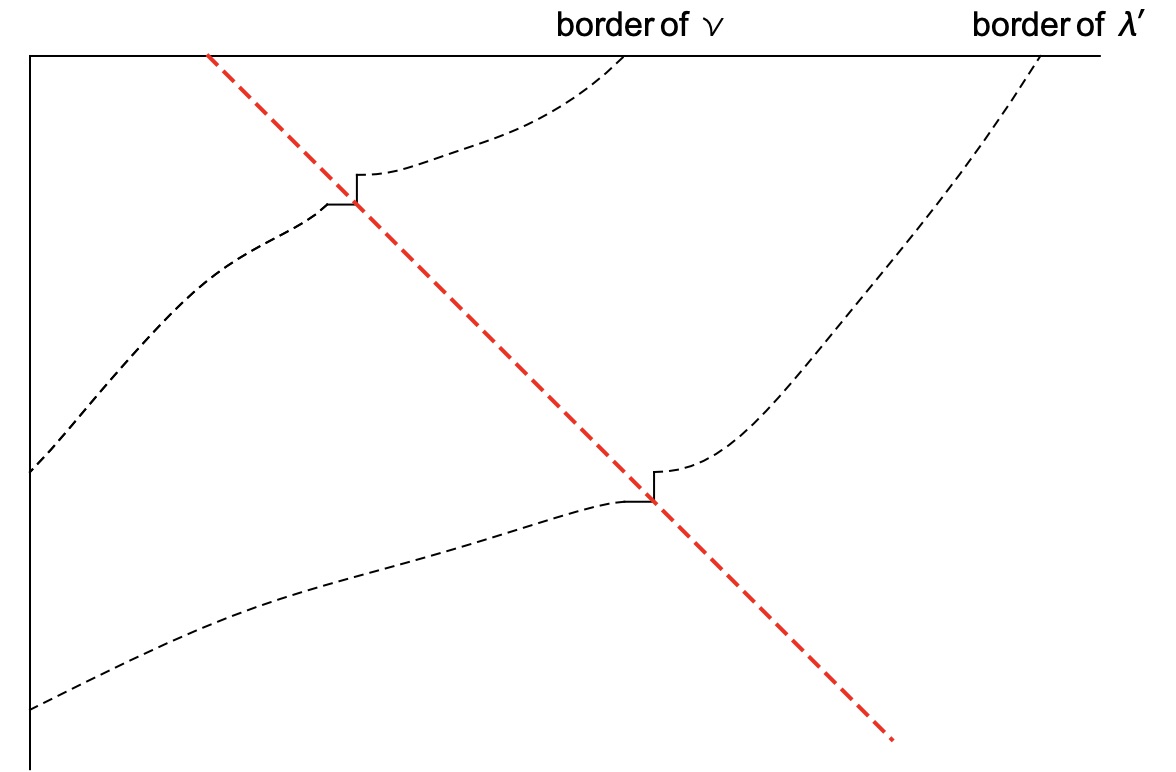} \caption{$\Delta_{\lambda'(U)\setminus\nu(U)}(\underline{k})=0$}
\label{nu-lambda-r}
\end{figure}
\begin{figure}[H]
\centering
\includegraphics[scale=0.24]{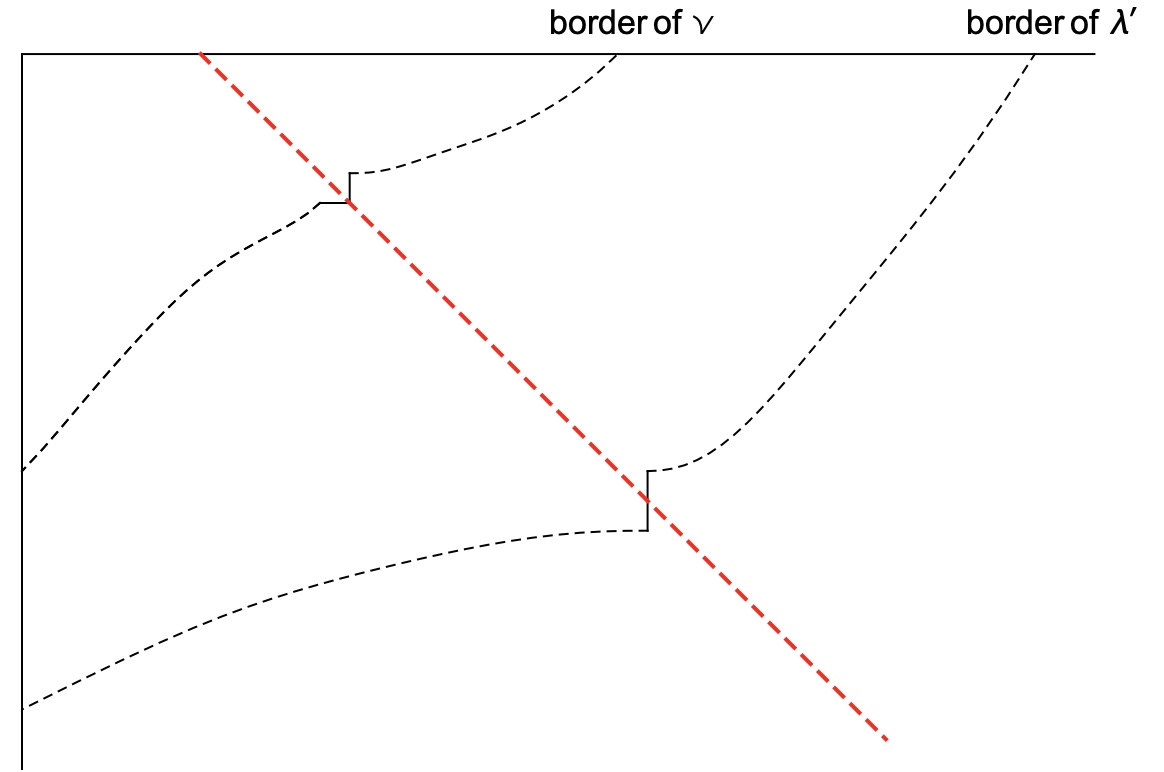} \caption{$\Delta_{\lambda'(U)\setminus\nu(U)}(\underline{k})=-1$}
\label{nu-lambda-0}
\end{figure}
The proof is completed.\hspace{.4cm}\myblacksquare

\subsection{Proof of Proposition \ref{prop:fusb2}}
\paragraph{1.} We prove the assertion (\ref{fubtve1}). Some gymnastics, \`a la in the Proof of Lemma \ref{psiind}, with the baxterized elements leads to the recursion
$$\mathring{\Psi}_{n;T_r,h}=\mathring{\Psi}_{n-1;T_r,h}\cdot \mathring{s}_{n}^{'\downarrow}\mathring{\mathfrak{d}}_n^{'\uparrow}
\mathring{\mathfrak{d}}_n^{\downarrow}\mathring{s}_{n}^{\uparrow}\ .$$
Here
$$\mathring{s}_{n}^{'\downarrow}:=s_{n-1,n}'(u_{n-1}+u_n)s_{n-2,n}'(u_{n-2}+u_n)\ldots s_{r+1,n}'(u_{r+1}+u_n)$$
and
$$\mathring{s}_{n}^{\uparrow}:=s_{r+1,n}(u_{r+1}+u_n)s_{r+2,n}(u_{r+2}+u_n)\ldots s_{n-1,n}(u_{n-1}+u_n)\ .$$
As before, it suffices to analyze the last evaluation $\ldots\big|_{u_n=c_n}$. Extracting the factors containing $u_n$ in the prefactor $\mathring{z}_{T;h}$,
see (\ref{nopref}), we see that we have to prove that for $U\nearrow T$ the evaluation of the expression
$$z_U^{'T} z_U^T E_U\cdot\mathfrak{X}\ \ ,\ \  \text{with}\ \ \mathfrak{X}:= \left( \mathring{s}_{n}^{'\downarrow}\mathring{\mathfrak{d}}_n^{'\uparrow}
\mathring{\mathfrak{d}}_n^{\downarrow}\mathring{s}_{n}^{\uparrow}\right)\vert_{u_{r+1}=c_{r+1},\dots ,u_{n-1}=c_{n-1}}\ ,$$
at $u_n=c_n$ is equal to $E_T$. Here $z_U^T$ is defined in (\ref{jztu}) and $z_U^{'T}$ is the following rational function in $u_n$:
$$z_U^{'T}:=\frac{u_n-h+\delta}{u_n+c_n-h}\ .$$
Lemma \ref{wacro2} implies that 
$$E_U\cdot  \left( \mathring{s}_{n}^{'\downarrow}\mathring{\mathfrak{d}}_n^{'\uparrow}\right)\vert_{u_{r+1}=c_{r+1},\dots ,u_{n-1}=c_{n-1}}=
\frac{h-u_n-x_n}{h-u_n-\delta}\,E_U\ .$$
By the fusion procedure of Theorem \ref{thm:fus}, the evaluation of  the expression 
$$z_U^T E_U\cdot \left(\mathring{\mathfrak{d}}_n^{\downarrow}\mathring{s}_{n}^{\uparrow}\right)\vert_{u_{r+1}=c_{r+1},\dots ,u_{n-1}=c_{n-1}}$$
at $u_n=c_n$ equals $E_T$, and $x_n E_T=c_n E_T$, so we are done.

\paragraph{2.} It is straightforward to see that the Jucys--Murphy elements are stable with respect to the anti-involution $\iota$ defined in (\ref{antiai}). 
Therefore the idempotents $E_T$ are stable with respect to $\iota$ as well. However, we have the identity
$$\iota \left( \mathring{\Psi}_{n;T_r,h} \right)=\mathring{\widetilde\Psi}_{n;T_r,h}$$
which implies the the assertion (\ref{fubtve2}).\hspace{.4cm}\myblacksquare

\vskip .3cm\noindent{\footnotesize
{\textbf{Acknowledgements.}
The work of D. B. was funded by Excellence Initiative of Aix-Marseille University -- A*MIDEX and Excellence Laboratory Archimedes LabEx, French "Investissements d'Avenir" programmes. The work of O. O. was supported by the Program of Competitive Growth of Kazan Federal University and by the grant RFBR 17-01-00585.}}

\end{document}